\documentclass[12pt,a4paper,reqno]{amsart}
\usepackage{mathrsfs}
\usepackage[all]{xy}
\usepackage{amssymb}
\usepackage{mathabx}
\title{Automorphisms of normal quasi-circular domains}
\author{Atsushi Yamamori}
\address{Department of Mathematics, Pohang University of Science and Technology, Pohang 790-784, The Republic of Korea}
\thanks{The research of the authors was supported by SRC-GaiA (The Center for Geometry and itsApplications), the Grant 2011-0030044 from The Ministry of Education, The Republic of Korea}
\email{yamamori@postech.ac.kr, ats.yamamori@gmail.com}
\theoremstyle{definition}
\newtheorem{theorem}{Theorem}[section]

\newtheorem{proposition}[theorem]{Proposition}

\newtheorem{remark}{Remark}
\newtheorem{problem}{Problem}
\newtheorem{example}{Example}[section]
\newtheorem{definition}{Definition}[section]
\begin{document}
\begin{abstract}
It was shown by Kaup that every origin-preserving automorphism of quasi-circular domains is a polynomial mapping.
In this paper, we study how the weight of quasi-circular domains and the degree of such automorphisms are related.
By using the Bergman mapping, we prove that every origin-preserving automorphism of  normal quasi-circular domains in $\mathbb C^2$ is linear.
\end{abstract}
\maketitle
\section{Introduction}
A domain $D$ is called a {\it quasi-circular} domain (or {\it $(m_1,\ldots, m_n)$-circular}) if it is invariant under the following mapping:
\[
D \ni (z_1,\ldots,z_n) \mapsto (e^{i m_1\theta}z_1,\ldots, e^{i m_n\theta} z_n) \in D, \mbox{ for some $m_1,\ldots,m_n\in \mathbb Z_{+}$}.
\]
The $n$-tuple $(m_1,\ldots, m_n)$ is called the {\it weight} of the quasi-circular domain $D$.
In particular, if $m_1=\cdots=m_n$, it is called circular.
For instance, the followings are examples of quasi-circular domains:
\begin{align*}
{\mathbb G}_2&=\{(z_1+z_2,z_1z_2)\colon |z_1|,|z_2|<1\},\\
\mathbb E
&=\{z\in\mathbb C^3; |z_1-\overline{z_2}z_3|+ |z_2-\overline{z_1}z_3| + |z_3|^2<1\}.
\end{align*}
These quasi-circular domains have been studied from various aspects (see, \cite{Agler2001,Agler2004,Edi,Jarnicki,Kos,Pflug,Young} and references therein).
The symmetrized bidisk $\mathbb G_2$ is known as an example of  a bounded pseudoconvex domain such that 
the Carath\'{e}odory and Kobayashi distances coincide, but it cannot be exhausted by domains biholomorphic to convex domains \cite{Costara2004,Edi2004}.
The domain $\mathbb G_2$ also appeared in \cite{Costara2005} in connection with the $2 \times 2$ spectral Nevanlinna-Pick problem.
Thus quasi-circular domains have received much interest in recent decades.

A famous theorem due to Cartan asserts that every origin-preserving automorphism of a circular domain
is linear.
By using Cartan's theorem, the automorphism groups of various circular domains
have been extensively investigated. 
It is also known that every origin-preserving automorphism of a quasi-circular domain
is a polynomial mapping \cite[Lemma 1]{Kaup}.
For quasi-circular cases, the following problem arises naturally:
\begin{problem}\label{problem}
Let $D$ be a quasi-circular domain and
suppose that $f$ is an origin-preserving automorphism of $D$.
Describe how the weight $(m_1,\ldots, m_n)$ of $D$ and the degree of the polynomial $f$ are related.
\end{problem}
In this paper, we focus our attention to one of the most interesting case \mbox{deg $f=1$}.
Our main result tells us that every origin-preserving automorphism of a normal quasi-circular domain in $\mathbb C^2$ is a linear mapping:
\begin{theorem}
Let $D, D'\subset \mathbb C^2$ be normal quasi-circular domains. 
Then a biholomorphism $f:D\rightarrow D'$ which preserves the origin is linear.
In particular, an automorphism $f \in \mbox{Aut}(D)$ which preserves the origin is linear.
\end{theorem}
This theorem is an answer of Problem \ref{problem} for $\deg f =1$ and $n=2$.

For circular domains, there is a commutative relation between $\rho_\theta(z)=e^{i \theta} z$ and every origin-preserving automorphism $\varphi$.
Namely we have $\varphi \circ \rho_\theta = \rho_\theta \circ \varphi$. For the proof of Cartan's theorem, this relation is substantial.
However we do not pursue this idea here (see Remark \ref{rem}).
Instead, we pursue an idea given by Ishi and Kai \cite{Ishi}.
The proof of our main theorem uses the Bergman mapping (also known as the Bergman coordinate).
This mapping was introduced by Stefan Bergman.
After its discovery, the Bergman mapping appeared in many studies (cf. \cite{G-K0,G-K,Ishi,Lu,Song,Tsuboi})
and it played a substantial role in their studies.
We will see that the Bergman mapping also plays an essential role in our study.

The organization of this paper is as follows.
Subsequent to this introduction, Section \ref{sec2} provides some basic properties of minimal and representative domains.
In Section \ref{sec3}, we introduce the normal quasi-circular domain and the Bergman mapping. We will see that the Bergman mapping is a linear mapping if a domain is normal quasi-circular.
This fact is used to prove our main result.
\section{Preliminaries}\label{sec2}
In this section, we introduce the minimal and representative domain and collect some basic facts.
Let us begin this section with some basic properties of the Bergman kernel.
\subsection{Basics of Bergman kernels}
Let $D$ be a domain in $\mathbb C^n$. Define the Bergman space $A^2(D)$ by
\begin{align*}
A^2(D)=\left\lbrace f\in \mathcal O(D); \int_D |f(z)|^2 dV(z) < \infty \right\rbrace.
\end{align*}
The Bergman space is a Hilbert space with the inner product
$$ \langle f,g \rangle = \int_{D} f(z)\overline{g(z)} dV(z) .$$
The reproducing kernel of the Bergman space is called the Bergman kernel and denote it by $K_D$.
Namely, the Bergman kernel is the unique function satisfying the following property:
\begin{align}
f(z)=\int_D f(w) K_D(z,w)dV(w), \quad \mbox{ for all } f\in A^2(D).
\end{align}
The Bergman kernel is also obtained by a complete orthonormal basis $\{ e_k \}_{k\in \mathbb N}$ of the Bergman space:
\begin{align}\label{CONS}
 K_D(z,w)=\sum_{k\in \mathbb N}e_k(z)\overline{e_k(w)}.
\end{align}
An important property of the Bergman kernel is relative invariance under the biholomorphisms.
Let $\varphi:D\rightarrow D'$ be a biholomorphism. Then the Bergman kernels $K_D$ and $K_{D'}$ satisfy the following relation:
\begin{align}\label{trans}
K_D(z,w)=\overline{\det J(\varphi, w)} K_{D'}(\varphi(z) ,\varphi (w)) \det J(\varphi, z).
\end{align}
Here $J(\varphi,z)$ is the Jacobian matrix of $\varphi = {}^t (\varphi_1, \ldots , \varphi_n)$ at $z$:
$$J(\varphi,z):=
\begin{pmatrix} 
\dfrac{\partial \varphi_1}{\partial z_1}(z)& \cdots & \dfrac{\partial \varphi_1}{\partial z_n}(z)
\\ \vdots & \ddots & \vdots\\
\dfrac{\partial \varphi_n}{\partial z_1}(z) & \cdots &\dfrac{\partial \varphi_n }{\partial z_n}(z) 
\end{pmatrix}.$$
The unit disk is an example whose Bergman kernel has an explicit form.
\begin{example}\label{disk}
For the unit disk $\mathbb D=\{|z|<1\} \subset \mathbb C$, 
the set $\{(\frac{k+1}{\pi} )^{\frac{1}{2}}z^k  \}_{k\in\mathbb Z_{+}}$ forms a complete orthonormal basis of the Bergman space
$A^2(\mathbb D)$. Using \eqref{CONS}, we compute the Bergman kernel:
\begin{align*}
K_{\mathbb D}(z,w)&=\dfrac{1}{\pi}\sum_{k=0}^\infty  (k+1)(z\overline{w})^k =\dfrac{1}{\pi(1-z\overline{w})^2 }.
\end{align*}
\end{example}
Further information about the Bergman kernel can be found in \cite{Greene,Krantz}.
\subsection{Minimal domains and Representative domains}
Recall that a bounded domain $D$ in $\mathbb C^n$ is called a minimal domain with a center at $z_0\in D$
if $\mbox{Vol}(D') \geq \mbox{Vol}(D)$ for any biholomorphism $\varphi: D\rightarrow D'$ such that $J(\varphi, z_0)=1$.
It is known that 
the minimality is equivalent to the following condition on the Bergman kernel \cite{Maschler}:
\begin{theorem}
A bounded domain $D$ is a minimal domain with the center at $z_0$, if and only if
\begin{align}
K_D(z,z_0)\equiv c, \quad \mbox{for any $z\in D$,}
\end{align}
where $c$ is a non-zero constant.
\end{theorem}
By the reproducing property of the Bergman kernel, we see that
\begin{align*}
1=\int_D 1 \cdot K(z,z_0)  dV(z)= c\int_D dV(z).
\end{align*}
Thus the constant $c$ must be equal to $1/\mbox{Vol}(D)$.

By Example \ref{disk},
we know that $K_{\mathbb D}(z,0)=1/\pi=1/\mbox{Vol} (\mathbb D) $.
Thus the unit disk is a minimal domain with the center at the origin.

For $z,w\in D$ such that $K_D(z,w)\not=0$,
we define an $n \times n$ matrix $T_D$ by
$$T_D (z,w):=
\begin{pmatrix} 
\dfrac{\partial^2 }{\partial \overline{w_1}\partial z_1}\log K_D(z,w) & \cdots & \dfrac{\partial^2 }{\partial \overline{w_1}\partial z_n}\log K_D(z,w)
\\ \vdots & \ddots & \vdots\\
\dfrac{\partial^2 }{\partial \overline{w_n}\partial z_1}\log K_D(z,w) & \cdots & \dfrac{\partial^2 }{\partial\overline{ w_n}\partial z_n}\log K_D(z,w)
\end{pmatrix}.$$
The matrix $T_D(z,z)$ is a positive definite hermitian matrix for all $z\in D$.
The matrix $T_D$ possesses the following transformation formula under the biholomorphisms:
\begin{align}\label{eq:T}
T_D(z,w)=\overline{ {}^t J(\varphi, w)} T_{D'}(\varphi(z) ,\varphi (w))  J(\varphi, z), \quad \mbox{if $K_D(z,w)\not=0$} .
\end{align}
By following a paper by Q.-K.~Lu \cite{Lu}, we introduce the representative domain.
\begin{definition}
A bounded domain $D$ in $\mathbb C^n$  is called a representative domain (in the sense of Q.-K.~Lu)
if there exists a point $z_0\in D$ such that $T_D(z,z_0)$ 
is a constant matrix for all $z\in D$. The point $z_0$ is called the center of the representative domain.
\end{definition}
We already know that the unit disk is minimal with the center at the origin.
Moreover it is also representative with the same center.
Indeed, by a simple computation, we have $T_{\mathbb D}(z,w)=2/(1-z\overline{w})^2$ and also $T_{\mathbb D}(z,0)=2$.

We finish this section with a remark on zeros of the Bergman kernel $K_D$.
\begin{remark}
The matrix $T_{D}(z,w)$ is not well-defined for $z,w\in D$ such that $K_D(z,w)=0$.
If a domain $D$ is a homogeneous bounded domain then it is known that $K_D(z,w)\not =0$ for all $z,w\in D$ (see \cite[Proposition 3.1]{Ishi}).
However there are non-homogeneous examples whose Bergman kernels are not zero-free.
For instance, the followings are such examples:
\begin{enumerate}
\item[(i)] $\{z\in \mathbb C; r<|z|<1\}$ for $r<e^{-2}$ (see \cite{S}),
\item[(ii)] $\{z\in\mathbb C^n; |z_1|+ \cdots + |z_n|<1 \}$ for $3\geq n$ (see \cite{Boas}),
\item[(iii)] $\left\lbrace (z_1,z_2)\in\mathbb C^2; |z_2|< \dfrac{1}{1+|z_1|}  \right\rbrace $ (see \cite{Boas2}).
\end{enumerate}
For further information about zeros of the Bergman kernel, see \cite{Ahn,Jarnicki2,Krantz2,Y,Y2} and references therein.
\end{remark}
\section{Bergman mapping for quasi-circular domains}\label{sec3}
The purpose of this section is to study the Bergman mapping for quasi-circular domains and prove our main theorems (Theorems \ref{main} and \ref{main0}).
We begin our study with some definitions.
In the following we only consider bounded domains which contain the origin.
\begin{definition}\label{def:quasi}
Let $ m_1,\ldots,m_n \in \mathbb Z_{+}$.
A bounded domain $D$ is called quasi-circular (or $(m_1,\ldots,m_n)$-circular)
if $(e^{i m_1\theta}z_1,\ldots, e^{i m_n\theta} z_n)\in D$ for any $\theta \in \mathbb R$ and
$(z_1,\ldots,z_n)\in D$.
The $n$-tuple $(m_1,\ldots, m_n)$ is called the weight of a quasi-circular domain.
\end{definition}
If $m_1=\cdots = m_n$, then it is an usual circular domain.
Now we define the normal quasi-circular domain.
\begin{definition}
Let $D\subset \mathbb C^2$ be a quasi-circular domain and $(m_1, m_2)$ its weight.
Without loss of generality we may assume that $m_1 \leq m_2 $ and $\mathrm{gcd}(m_1, m_2)=1$.
A quasi-circular domain $D$ is called normal if $m_1\geq2$.
\end{definition}
Let us give some concrete examples.
Consider the following two domains
\begin{align*}
D_1&=\{(z_1, z_2)\in \mathbb B^2; | z_1^3+z_2^2 |<1\},\\
D_2&=\{(z_1, z_2)\in \mathbb B^2; | z_1^2+z_2 |<1\}.
\end{align*}
Then $D_1$ is a $(2,3)$-circular domain and also normal.
On the other hand, $D_2$ is not normal, since $D_2$ is a $(1,2)$-circular domain.
It is easy to see that a $(p_1,p_2)$-circular domain is normal for any prime numbers $p_1, p_2$ such that $p_1 < p_2$.

For the first step of our study,
we  prove the minimality of quasi-circular domains.
\begin{proposition}\label{minimal}
If a domain $D\subset \mathbb C^2$ is quasi-circular, then it is a minimal domain with the center at the origin.
\end{proposition}
\begin{proof}
Define $f_{\theta}: D\rightarrow D$ by $f_{\theta}(z_1, z_2)=(e^{i m_1\theta}z_1, e^{i m_2\theta} z_2)$
for $z\in D$ and $\theta \in \mathbb R$.
Then the Jacobian matrix $J(f_\theta,z)$ is given by $\mathrm{diag}(e^{im_1 \theta},e^{im_2 \theta} )$.
Since $D$ is quasi-circular, $f_{\theta}$ is an automorphism of $D$. By the relative invariance of the Bergman kernel \eqref{trans}
we obtain
\begin{align}\label{eq:Bergman}
K_{D}(z,0)=K_D(f_\theta(z),0),
\end{align}
for any $\theta\in \mathbb R$.
Using \eqref{eq:Bergman} and the Taylor expansion, we have
\begin{align*}
K_D(z,0)&=\sum_{k\in\mathbb Z^2_{\geq 0}} a_k z_1^{k_1} z_2^{k_2},\\
&=\sum_{k\in\mathbb Z^2_{\geq 0}} e^{i(\sum_{j=1}^2 m_j k_j)\theta } a_k   z_1^{k_1} z_2^{k_2}=K_D(f_\theta(z),0).
\end{align*}
It follows that $a_k= e^{i(\sum_{j=1}^2 m_j k_j)\theta } a_k$ for any $\theta \in \mathbb R$ and $k \in \mathbb Z_{\geq 0}^2$.
Since $\sum_{j=1}^2 m_j k_j \not =0$ except $(k_1,k_2)=(0,0)$, all coefficients except the constant term are zero.
Thus $K_{D}(z,0)$ is a constant.
Since $D$ is bounded, we know that $K(0,0)$ is non-zero constant.
Therefore, we finally conclude that $K_{D}(z,0)$ is a non-zero constant.
\end{proof}
For the minimality of quasi-circular domains, we do not need to assume the normality for them.
On the other hand, we need the normality to prove that a quasi-circular domain is representative.
\begin{proposition}\label{repre}
Let $D \subset \mathbb C^2$ be a normal quasi-circular domain.
Then $D$ is a representative domain with the center at the origin.
\end{proposition}
\begin{proof}
In the following, for simplicity, we use the notation 
$$K_{\overline{i}j}(z,w)= \frac{\partial^2}{\partial{\overline{w}_i}\partial{z_j}  }\log K_D(z,w).$$
By the transformation formula \eqref{eq:T}, we have
\begin{align}\label{TD}
\begin{pmatrix}
K_{\overline{1}1}(z,0) & K_{\overline{1}2}(z,0)\\
K_{\overline{2}1}(z,0) & K_{\overline{2}2}(z,0)
\end{pmatrix}
=
\begin{pmatrix}
K_{\overline{1}1}(f_\theta(z),0) & e^{i(m_2-m_1)\theta}  K_{\overline{1}2}(f_\theta(z),0)\\
e^{i(m_1-m_2)\theta} K_{\overline{2}1}(f_\theta(z),0) & K_{\overline{2}2}(f_\theta(z),0)
\end{pmatrix}.
\end{align}
for any $\theta\in \mathbb R$.
By a similar argument used in Proposition \ref{minimal}, we know that $K_{\overline{k}k}(z,0)$ is a constant
for $k=1,2$. 
Let us prove that $K_{\overline{1}2}(z,0)$ is a constant.
By \eqref{TD} and the Taylor expansion, we see that
\begin{align*}
K_{\overline{1}2}(z,0)&= \sum_{k\in\mathbb Z^2_{\geq 0}} a_k z_1^{k_1} z_2^{k_2},\\
&=e^{i(m_2-m_1)\theta}  \sum_{k\in\mathbb Z^2_{\geq 0}} e^{i(\sum_{j=1}^2 m_j k_j)\theta }   a_k z_1^{k_1} z_2^{k_2},\\
&= e^{i(m_2-m_1)\theta}  K_{\overline{1}2}(f_\theta(z),0).
\end{align*}
Then we obtain $a_k=e^{i(m_2-m_1 + \sum_{j=1}^2 m_j k_j  )\theta } a_k$.
Since $m_2-m_1>0$,
$c_{k,m}=m_2-m_1 + \sum_{j=1}^2 m_j k_j$ is non-zero for any $k_1,k_2 \geq 0$.
Thus we know that $a_k=0$ for any $k_1,k_2 \geq 0$ and $K_{\overline{1}2}(z,0)\equiv 0$. 
Next we consider $K_{\overline{2}1}(z,0)$. A similar argument shows that 
\begin{align*}
K_{\overline{2}1}(z,0)&= \sum_{k\in\mathbb Z^2_{\geq 0}} a'_k z_1^{k_1} z_2^{k_2},\\
&=e^{i(m_1-m_2)\theta}  \sum_{k\in\mathbb Z^2_{\geq 0}} e^{i(\sum_{j=1}^2 m_j k_j)\theta }   a'_k z_1^{k_1} z_2^{k_2},\\
&= e^{i(m_1-m_2)\theta}  K_{\overline{2}1}(f_\theta(z),0).
\end{align*}
It follows that $a'_k=e^{i(m_1-m_2 + \sum_{j=1}^2 m_j k_j  )\theta } a'_k$.
For any $k_1\geq 0, k_2\geq 1$, the number 
$c'_{k,m}=m_1-m_2 + \sum_{j=1}^2 m_j k_j$ is a non-zero constant.
It follows that $a'_k=0$ for $k_1\geq 0, k_2 \geq 1$.

There remains the task of proving that $c'_{k,m}\not =0$ for $k_1\geq 1$ and $k_2=0$.
In this case, $c'_{k,m}=(k_1+1)m_1-m_2$.
If $c'_{k,m}=0$, then $(k_1+1)m_1=m_2$.
On the other hand, by the condition $\mathrm{gcd}(m_1,m_2)=1$, we have $n m_1\not = m_2$ for any $n \geq 2$.
It is contradiction. Thus we conclude that $a'_k=0$ except for the case $(k_1 ,k_2)=(0,0) $. It follows that $K_{\overline{2}1}(z,0)$ is a constant.
Hence $T_D(z,0)$ is a constant matrix.
\end{proof}
We note that the constant term of $K_{\overline{2}1}(z,0)$ also vanishes. In other words, $T_D(z,0)$ is a diagonal matrix.

Now we introduce the Bergman mapping and recall its relevant properties.
\begin{definition}
Let $D$ be a bounded domain.
Set $U_p^D:=\{ z\in D; K_D(z,p)\not=0 \}$. Define a mapping $\sigma_p^D:U_p^D \rightarrow \mathbb C^n$ by
\begin{align}
\sigma_{p}^D (z):=\left. T_D(p,p)^{-1/2}\mathrm{grad}_{\overline{w}}\log \dfrac{K_D(z,w)}{K_D(p,w) } \right|_{w=p} ,\quad \mbox{for } z\in U_p^D.
\end{align}
The mapping $\sigma_p^D$ is called the Bergman mapping defined at $p$.
\end{definition}
Here we set
$$\mathrm{grad}_{\overline{w}}  f(w) := {}^{t} \left( \dfrac{\partial f}{\partial \overline{w_1}} (w), \ldots,
\dfrac{\partial f}{\partial \overline{w_n}}(w)  \right),$$
for anti-holomorphic functions $f$ on $D$.

By the definition of the Bergman mapping, we easily verify the following properties (see also \cite{Ishi}):
\begin{align}
 \sigma_{p}^D(p)&=0, \label{property1}\\
 J(\sigma_p^D,z)&=T_D(p,p)^{-1/2} T_D(z,p) \quad \mbox{for } z\in U_p^D.\label{property2}
\end{align}
Let $D\rightarrow D'$ be a biholomorphism. Define an $n\times n$ matrix $L(\varphi, p)$ by
$$L(\varphi, p):= T_{D'}(\varphi(p),\varphi(p) )^{-1/2} \overline{{}^t J(\varphi, p) ^{-1}} T_D(p,p)^{1/2}.$$
The matrix $L(\varphi, p)$ is a unitary matrix. Indeed, by \eqref{eq:T}, we have
\begin{align*}
L(\varphi, p)^* L(\varphi, p)&=
T_{D}(p,p )^{-1/2} ( J(\varphi, p) ^{-1}  T_{D'}(\varphi(p),\varphi(p) )^{-1} \overline{{}^t J(\varphi, p) ^{-1}} )T_D(p,p)^{1/2},\\
&=T_{D}(p,p )^{-1/2} T_D(p,p ) T_{D}(p,p )^{1/2},\\
&=I.
\end{align*}
The following relation between the Bergman mapping $\sigma^D_p$ and the unitary matrix $L(\varphi,p)$ is substantial for our purpose (see \cite{Ishi}).
\begin{proposition}\label{diagram}
Let $D, D'$ be bounded domains and $\varphi: D\rightarrow D'$ a biholomorphism.
One has $\sigma_{\varphi(p)}^D \circ \varphi= L(\varphi,p) \circ \sigma_p^D$.
In other words, the diagram
\[\xymatrix{
U_p^D \ar[d]_{\sigma_{p}^D} \ar[r]^\varphi_\sim \ar@{}[dr]|\circlearrowright & U_{\varphi(p)}^{D'} \ar[d]^{\sigma_{\varphi(p)}^{D'}} \\
\mathbb C^n \ar[r]^{L(\varphi,p)} & \mathbb C^n \\
}
\]
is commutative.
\end{proposition}

In the following, we consider minimal representative domains.
In \cite[Corollary 2]{Lu}, Q.-K.~Lu proved that if $D$ and $D'$ are both representative domains in $ \mathbb C^n$ then any biholomorphism which
maps the center of $D$ to that of $D'$ is an affine transformation.
If two domains $D,D'$ are minimal representative domains with the center at the origin, then we obtain the following:
\begin{proposition}\label{comm}
Assume that $D$ and $D'$ are minimal representative domains with the center at the origin in $ \mathbb C^n$.
Then any biholomorphism which maps the center of $D$ to that of $D'$ is linear.
\end{proposition}
\begin{proof}
Since two domains $D$ and $D'$ are both minimal, we know that $K_{D}(z,0)\not =0$ and $K_{D'}(z',0)\not =0$
for all $z\in D$ and $z' \in D'$. It follows that $U_{0}^{D}=D$ and $U_{0}^{D'}=D'$.
Next we prove that the Bergman mapping $\sigma_{0}^D$ is linear.
Since $D$ is representative, we know that $T_D(z,0)\equiv T_D(0,0)$.
This, together with \eqref{property2}, implies that $J(\sigma_{0}^D, z )=T_D(0,0)^{1/2}$.
Thus the Jacobian matrix $J(\sigma_{0}^D, z)$ is a constant matrix. It follows that $\sigma_{0}^D$
is an affine transformation. By \eqref{property1}, we have $\sigma_{0}^D(z)=T_D(0,0)^{1/2}z$. By the same argument, $\sigma_{0}^{D'}$ is also linear.

Let $f: D \rightarrow D' $ be a biholomorphism which maps the center of $D$ to that of $D'$.
The above argument and Proposition \ref{diagram} give us the following commutative diagram:
\[\xymatrix{
D \ar[d]_{T_{D}(0,0)^{\frac{1}{2}} =\sigma_{0}^D} \ar[r]^f_\sim \ar@{}[dr]|\circlearrowright & D' \ar[d]^{ \sigma_{0}^{D'}= T_{D'}(0,0)^{\frac{1}{2}} } \\
\mathbb C^n \ar[r]^{L(f,0)} & \mathbb C^n. \\
}
\]
Therefore we conclude that $f(z)=T_{D'}(0,0)^{-\frac{1}{2}} L(f, 0) T_{D}(0,0)^{\frac{1}{2}}z$.
It is obviously linear.
\end{proof}
This proposition is a natural generalization of the argument given in \cite[Section 2.2]{Ishi}
to the minimal representative domains.
Now we are ready to prove our main result.
\begin{theorem}\label{main}
Let $D, D'\subset \mathbb C^2$ be normal quasi-circular domains and
$f:D\rightarrow D'$ an origin-preserving biholomorphism.
Then the mapping $f$ is given by $f=T_{D'}(0,0)^{-\frac{1}{2}} L(f, 0) T_{D}(0,0)^{\frac{1}{2}}$.
In particular, the mapping $f$ is a linear mapping.
\end{theorem}
\begin{proof}
Since the second part of the theorem follows from the first part, we prove only the first part of the theorem.
By Propositions \ref{minimal} and \ref{repre}, we know that $D, D'$ are minimal representative domains with the center at the origin.
Then we obtain the following commutative diagram by Proposition \ref{comm}:
\[\xymatrix{
D \ar[d]_{T_{D}(0,0)^{\frac{1}{2} }= \sigma_0^D } \ar[r]^f_\sim \ar@{}[dr]|\circlearrowright & D' \ar[d]^{\sigma_0^{D'}=T_{D'}(0,0)^{\frac{1}{2}} } \\
\mathbb C^2 \ar[r]^{L(f,0)} & \mathbb C^2. \\
}
\]
Hence we conclude that $f(z)=T_{D'}(0,0)^{-\frac{1}{2}} L(f, 0) T_{D}(0,0)^{\frac{1}{2}}z$.
\end{proof}
As a special case of this theorem we obtain the following.
\begin{theorem}\label{main0}
Let $D$ be a normal quasi-circular domain in $\mathbb C^2$ and $f$ an origin-preserving automorphism of $D$.
Then we have $f=T_{D}(0,0)^{-\frac{1}{2}} L(f, 0) T_{D}(0,0)^{\frac{1}{2}}$.
In particular, the mapping $f$ is a linear mapping.
\end{theorem}

We should note that there is a concrete example of $(1,2)$-circular domain
whose automorphism group contains an origin-preserving automorphism which is not linear.
In \cite{Zapa}, Zapa{\l}owski studied proper holomorphic self-mappings for the symmetrized $(p,n)$-ellipsoid $\mathbb E_{p,n}:=\pi_n (\mathbb B_{p,n} )$ where $\pi_n=(\pi_{n,1}, \ldots, \pi_{n,n} )$ and $\mathbb B_{p,n}$ are defined by
\begin{align*}
\mathbb B_{p,n}&:= \left\lbrace z\in \mathbb C^n: \sum_{j=1}^n |z_j|^{2p}<1 \right\rbrace,\\
\pi_{n,k}(z)&:= \sum_{1\leq j_1 < \cdots <j_k\leq n} z_{j_1} \cdots z_{j_k}, \quad 1\leq k \leq n,\quad z=(z_1,\ldots, z_n)\in \mathbb C^n.
\end{align*}
The symmetrized $(p,n)$-ellipsoid $\mathbb E_{p,n}$ is a $(1,2,\ldots, n)$-circular domain.
In the same paper, Zapa{\l}owski also determined the automorphism group of $\mathbb E_{p,n}$.
For instance, if $(p,n)=(1/2,2)$ then the automorphism group of $\mathbb E_{1/2,2}$ contains the following mapping:
\begin{align*}
\varphi(z_1, z_2)=\left(\zeta z_1 , \zeta^2 \left(\frac{z_1^2}{4} -z_2 \right)\right), \quad (z_1, z_2) \in \mathbb E_{1/2,2},
\end{align*}
where $\zeta \in \mathbb T=\{z\in\mathbb C; |z|=1\}$. Obviously, $\varphi$ is a non-linear mapping which preserves the origin.
Thus we cannot drop the condition ``$m_1 \geq 2$" in the definition of the normality.

We conclude this paper with some remarks.
\begin{remark}
We showed that every normal quasi-circular domain is a representative domain which is also minimal with the same center.
One may expect that these two kinds of centers coincide with each other whenever a domain is minimal and representative.
However this expectation is not true in general.
Actually, Maschler proved that there are representative domains which are not minimal domains with the same center \cite[Corollary 2]{Maschler}.
\end{remark}
\begin{remark}
In \cite{Ishi}, Ishi and Kai defined the representative domain as the image of the Bergman mapping.
In their definition, the representative domain $D$ always satisfies $T_D(z,0) \equiv I$. Namely it is a representative domain with the center at the origin in the sense of Q.-K. Lu.
As we have seen in this section, for the linearity of the Bergman mapping, the condition $T_D(z,0)\equiv T_D(0,0)$ is essential.
This is why we use Q.-K. Lu's definition in this paper.
\end{remark}
\begin{remark}\label{rem}
Let $D$ be a circular domain in $\mathbb C^2$ and put $\rho_\theta(z_1,z_2)=(e^{i\theta} z_1,e^{i\theta} z_2 )$.
Since every circular domain is a minimal representative domain with the center at the origin, we can prove Cartan's theorem by using the Bergman mapping \cite{Ishi}.
For circular cases, there is another way to prove Cartan's theorem.
Namely, for circular cases, we have a relation $\varphi \circ \rho_\theta = \rho_\theta \circ \varphi$ for any $\varphi \in \mbox{Aut}(D)$ such that $\varphi(0)=0$. By using this relation, we can conclude that $f$ must be linear (cf. \cite[Chapter 6]{Gong} or \cite[Chapter 5]{Nara}).
On the contrary to the circular cases, we cannot obtain such a relation for $f_\theta(z_1,z_2)=(e^{im_1\theta} z_1,e^{im_2\theta} z_2 )$ by using Cartan uniqueness theorem.
Put $g=f_{-\theta}  \circ \varphi^{-1}  \circ f_\theta \circ \varphi$. In this case, the matrix $J(f_{\theta},z)$ is an element of the center of
$\mbox{Mat$_{2\times 2}(\mathbb C)$}$ if and only if $m_1=m_2$.
In other words, if $m_1\not=m_2$, then the matrix $J(f_{\theta},z)$ does not commute with all elements of $\mbox{Mat$_{2\times 2}(\mathbb C)$}$.
Thus we cannot conclude that $J(g,z)=I$ in the same way as for the circular domains. 
\end{remark}
\section*{Acknowledgement}
The author would like to thank Professor Kang-Tae Kim for helpful discussion.
The author also thanks to Hyeseon Kim and Van Thu Ninh for their comments on
an earlier draft of the manuscript.
Moreover the author is indebted to Pawe{\l} Zapa{\l}owski, who pointed out an error in an earlier version of the manuscript.
\bibliographystyle{amsplain}

\end{document}